\documentclass[11pt]{article}

\usepackage[margin=1in]{geometry}
\usepackage{amsmath,amssymb,amsthm,mathtools}
\usepackage{enumitem}
\usepackage{microtype}
\usepackage[colorlinks=true,linkcolor=blue,citecolor=blue,urlcolor=blue]{hyperref}

\numberwithin{equation}{section}

\newtheorem{theorem}{Theorem}[section]
\newtheorem{lemma}[theorem]{Lemma}

\newcommand{\norm}[1]{\left\lVert #1\right\rVert}
\newcommand{\abs}[1]{\left\lvert #1\right\rvert}

\title{The Numerical Index of Two-Dimensional Real $\ell_p$ Spaces}
\author{Rafael Chiclana\thanks{Department of Mathematics, Michigan State University. 
		Email: \texttt{chiclan1@msu.edu}.}}
\date{}

\begin{document}
\maketitle

\begin{abstract}
The computation of the numerical index of classical Banach spaces is one of
the original problems in the theory.  In this paper, we compute the numerical
index of two-dimensional real \(\ell_p\)-spaces for all \(p\ge1\).  More
precisely, we prove that
\[
n(\ell_p^2)=v(J),
\qquad
J=
\begin{pmatrix}
	0&1\\
	-1&0
\end{pmatrix},
\]
confirming the conjectured formula in the two-dimensional real case.
\end{abstract}

\section{Introduction}

The numerical index is a classical isometric invariant of Banach spaces,
introduced to quantify the relation between the norm and the numerical radius
of bounded linear operators.  Let \(X\) be a Banach space over the real field,
let \(X^*\) denote its dual space, and let \(\mathcal L(X)\) denote the space
of all bounded linear operators from \(X\) into itself.  For
\(T\in\mathcal L(X)\), the numerical radius of \(T\) is defined by
\[
v(T)
:=
\sup\bigl\{
|x^*(Tx)|:
x\in X,\ x^*\in X^*,\ \|x\|=\|x^*\|=x^*(x)=1
\bigr\}.
\]
The numerical radius is a seminorm on \(\mathcal L(X)\) and satisfies
\(v(T)\le \|T\|\) for every \(T\in\mathcal L(X)\).  The numerical index of
\(X\) is the number
\[
n(X)
:=
\inf\bigl\{
v(T):T\in\mathcal L(X),\ \|T\|=1
\bigr\}.
\]

The concept was introduced in the seminal work of Duncan, McGregor, Pryce and
White \cite{DuncanMcGregorPryceWhite1970}, after a question of Lumer, and was
developed systematically in the monographs of Bonsall and Duncan
\cite{BonsallDuncan1971,BonsallDuncan1973}.  One of the original problems in the theory was to compute the numerical index of classical \(L_p\)-spaces. In the real case, this problem remains
open for \(1<p<\infty\), \(p\ne2\). In the two-dimensional case, it was
conjectured that the numerical index of \(\ell_p^2\) is attained by the
rotation by \(90^\circ\), which corresponds to the operator
\[
J=
\begin{pmatrix}
	0&1\\
	-1&0
\end{pmatrix}.
\]
In this paper, we prove the conjecture.
\begin{theorem}\label{theo: main}
	For every \(p\ge1\),
	\[
	n(\ell_p^2)=v(J).
	\]
\end{theorem}

The cases \(p=1\) and \(p=2\) are classical. Moreover, as recalled in
\cite[Section~1]{MeriQuero2024}, the range \(1<p<2\) reduces by duality to
the range \(p>2\). Thus, in the proof, we always work under the assumption \(p>2\).

\subsection*{Related work}

Several general facts about the numerical index are classical.  In the real
case, Hilbert spaces of dimension at least two have numerical index zero, while
\(L_1(\mu)\)-spaces and \(C(K)\)-spaces have numerical index one.  For
\(1<p<\infty\), \(p\ne2\), the situation is much more subtle.  The exact value
of \(n(L_p(\mu))\) is not known in general, and even finite-dimensional cases
have required separate arguments.

For $1<p<\infty$, set $M_p:=v(J)$ on $\ell_p^2$. For the two-dimensional problem, Martín and Merí \cite{MartinMeri2009} proved general estimates for \(n(\ell_p^2)\).   If \(q=p/(p-1)\) is the
conjugate exponent, then
\[
\max\{2^{-1/p},2^{-1/q}\}M_p
\le n(\ell_p^2)\le M_p,
\]
and \(M_p=M_q\) by duality.  Since the upper estimate is realized by the rotation \(J\), this
led to the conjecture that equality should hold for every \(1<p<\infty\).

The conjecture was subsequently verified in several ranges of \(p\).  Merí and
Quero \cite{MeriQuero2021} studied numerical indices for absolute and
symmetric norms on the plane and proved, as a consequence, that
\[
n(\ell_p^2)=M_p
\qquad
\text{for } \frac32\le p\le3.
\]
Monika and Zheng \cite{MonikaZheng2023} refined these methods and proved the
equality in the larger interval
\[
1+\alpha_0\le p\le \alpha_1,
\]
where \(\alpha_0\approx0.4547\) and \(\alpha_1\) is determined by
\[
\frac1{1+\alpha_0}+\frac1{\alpha_1}=1.
\]
More recently, Merí and Quero \cite{MeriQuero2024} used Riesz--Thorin
interpolation to extend the equality for
\[
\frac65\le p\le\frac32
\qquad\text{and}\qquad
3\le p\le6.
\]
Theorem~\ref{theo: main} completes the computation of the numerical index of
real two-dimensional \(\ell_p\)-spaces.

We briefly describe the organization of the paper.  In
Section~\ref{sec:notation}, we introduce the preliminary results, notation,
and constants used in the proof. Section~\ref{sec:skeleton} contains the main auxiliary lemmas and
explains how they imply Theorem~\ref{theo: main}.  The proof of
Theorem~\ref{theo: main} is then given in Section~\ref{sec:4}.  The remaining
sections are devoted to the proofs of the auxiliary lemmas stated in
Section~\ref{sec:skeleton}.

\section{Preliminaries}\label{sec:notation}

We shall use the following formula for the numerical radius of an operator on
\(\ell_p^2\).

\begin{lemma}[Lemma 1 in \cite{MeriQuero2024}]\label{lem:numerical-radius-formula}
	Let \(1<p<\infty\) and $T=
	\begin{pmatrix}
		a&b\\
		c&d
	\end{pmatrix}
	$ be an operator on \(\ell_p^2\). Then
	\[
	v(T)
	=
	\max\left\{
	\max_{0\le t\le1}
	\frac{\abs{a+dt^p}+\abs{bt+ct^{p-1}}}{1+t^p},
	\,
	\max_{0\le t\le1}
	\frac{\abs{d+at^p}+\abs{ct+bt^{p-1}}}{1+t^p}
	\right\}.
	\]
\end{lemma}

We shall also use the following elementary invariance property (see \cite[Section~1]{MeriQuero2024}). If
\(U\colon X\to X\) is a surjective linear isometry, then for every
\(T\in\mathcal L(X)\),
\begin{equation}\label{eq:invariance}
\|U^{-1}TU\|=\|T\|,
\qquad
v(U^{-1}TU)=v(T).
\end{equation}
Define
\[
 L(t):=\frac{t^{p-1}}{1+t^p},\qquad
 M(t):=\frac{t}{1+t^p},\qquad
 N(t):=\frac{1-t^p}{1+t^p},
\]
and
\[
 R(t):=M(t)-L(t)=\frac{t-t^{p-1}}{1+t^p},
 \qquad 0\le t\le1.
\]
For $p>2$, a standard calculus argument shows that the function $R$ has a unique maximizer in $(0,1)$, which we denote by $t_0$.  We set
\[
 L_0:=L(t_0),\qquad
 M_0:=M(t_0),\qquad
 N_0:=N(t_0),\qquad
 R_0:=R(t_0)=v(J).
\]
By standard symmetry reductions, which are carried out
in the proof of Theorem~\ref{theo: main}, it is enough to study
operators of the form
\begin{equation}\label{eq: T}
T=
\begin{pmatrix}
	a & 1\\
	-(1-c) & -d
\end{pmatrix},
\qquad
a,d\ge0,
\quad
c\in[0,1].
\end{equation}
For convenience, we introduce the parameters
\[
\alpha:=\frac{a+d}{2},
\qquad
\beta:=\frac{d-a}{2},
\qquad
\eta:=\frac{p-2}{p}.
\]
We also define
\begin{equation}\label{eq:Hacd}
	H(a,c,d):=
	\max\left\{
	cL_0+\abs{\beta-\alpha N_0},
	\,
	\abs{\beta+\alpha N_0}-cM_0
	\right\}.
\end{equation}

\section{The main estimates}\label{sec:skeleton}

The proof of Theorem \ref{theo: main} is based on two estimates.  Our goal is to show that
\[
\frac{v(T)}{\norm{T}}\ge v(J)
\]
for every matrix \(T\) of the form \eqref{eq: T}, where $J$ is the rotation by $90^\circ$.  The argument separates into a lower
bound for the numerical radius and an upper bound for the operator norm.

The first estimate gives a quantitative lower bound for \(v(T)\) by evaluating the numerical
radius at the point where \(J\) attains its numerical radius.

\begin{lemma}\label{lem: lower bound}
Let $a,d\ge0$, $c\in[0,1]$, and $T$ defined as in (\ref{eq: T}). Then
\[
 v(T)\ge v(J)+H(a,c,d).
\]
\end{lemma}
Since $H(a,c,d)\geq 0$, Lemma~\ref{lem: lower bound} shows that, among matrices of the form
\eqref{eq: T}, the numerical radius is minimized by the rotation \(J\).
Moreover, \(H(a,c,d)\) provides a quantitative estimate of how the numerical
radius increases as the parameters \(a,d,c\) move away from the rotational case. A key feature of \(H\), which follows directly from its definition in
\eqref{eq:Hacd}, is the homogeneity
\begin{equation}\label{eq: homogeneity of H}
H(\lambda a,\lambda c,\lambda d)
=
\lambda H(a,c,d),
\qquad
\lambda\ge0,
\end{equation}
which will play an important role in the proof of the main theorem.

The second estimate provides an upper bound for the operator norm of $T$.

\begin{lemma}\label{lem: upper bound}
	Let $a,d\ge0$, $c\in[0,1]$, and $T$ defined as in (\ref{eq: T}). Then, 
	\[
	\norm{T}\le \frac{v(J)+H(a,c,d)}{v(J)}.
	\]
\end{lemma}
Combining Lemmas~\ref{lem: lower bound} and~\ref{lem: upper bound} yields
\[
\frac{v(T)}{\norm{T}}
\ge 
v(J)
\]
for every matrix \(T\) of the form \eqref{eq: T}. Since standard symmetry and
normalization arguments reduce the problem to this class of matrices,
Theorem~\ref{theo: main} follows.

In order to prove Lemma \ref{lem: upper bound}, we will show that it suffices to analyze the case $c=1$, which corresponds to an upper triangular matrix.  Indeed, we can express $T$ as the convex combination
\[
T=(1-c)J+cS,
\]
where
\[
S=
\begin{pmatrix}
	a/c & 1\\
	0 & -d/c
\end{pmatrix}.
\]
Then, the convexity of the operator norm together with the homogeneity of $H(a,c,d)$ given in (\ref{eq: homogeneity of H}) can be used to reduce the problem to estimating the norm of \(S\).

\begin{lemma}\label{lem: upper triangular}
Let $A,B\ge0$, and consider
\[
 S=\begin{pmatrix}A&1\\0&-B\end{pmatrix}.
\]
Then
\[
 \norm{S}\le \frac{v(J)+H(A,1,B)}{v(J)}.
\]
\end{lemma}

\section{Proof of Theorem \ref{theo: main}}\label{sec:4}
	
	Theorem~\ref{theo: main} follows from Lemmas~\ref{lem: lower bound} and~\ref{lem: upper bound}
	once we show that it is enough to study operators of the form \eqref{eq: T}.
	
\begin{proof}[Proof of Theorem~\ref{theo: main}]
	We first reduce to the case \(p>2\). The cases \(p=1\) and \(p=2\) are
	classical: \(n(\ell_1^2)=v(J)=1\), while \(n(\ell_2^2)=v(J)=0\). Now let
	\(1<p<\infty\), \(p\ne2\), and let \(q=p/(p-1)\). Since
	\((\ell_p^2)^*=\ell_q^2\), the identity \(v(T^*)=v(T)\) gives
	\(n(\ell_p^2)=n(\ell_q^2)\). Moreover \(v_{\ell_p^2}(J)=v_{\ell_q^2}(J)\)
	\cite[Section~1]{MeriQuero2024}. Thus the range \(1<p<2\) follows from the
	range \(q>2\), and it remains to prove the result for \(p>2\).
	By the reduction used by Merí and Quero in the proof of \cite[Theorem~3]{MeriQuero2024}, it is enough to consider operators of the form
	\[
	T=
	\begin{pmatrix}
		a&b\\
		-c&-d
	\end{pmatrix},
	\qquad
	a,b,c,d\ge0.
	\]
	Next, we note that we may assume $b\geq c$. Indeed, consider the isometry
	\[
	S=
	\begin{pmatrix}
		0&1\\
		1&0
	\end{pmatrix}.
	\]
	By the invariance property (\ref{eq:invariance}), replacing \(T\) by \(-S^{-1}TS\) does not change the
	quotient \(v(T)/\|T\|\). A direct computation gives
	\[
	-S^{-1} T S
	=
	\begin{pmatrix}
		d&c\\
		-b&-a
	\end{pmatrix}.
	\]
	Thus this replacement preserves the class of matrices under consideration and
	interchanges the roles of \(b\) and \(c\). Next, by a density argument, it suffices to consider the case \(b>0\).  Indeed, if
	\(b=0\), then the assumption \(b\ge c\) forces \(c=0\). Since we only need
	to consider nonzero operators, we may also assume that \(a\) and \(d\) are not
	both zero. For $b_n>0$ with $b_n \to 0$, set
	\[
	T_n =
	\begin{pmatrix}
		a&b_n\\
		0&-d
	\end{pmatrix}.
	\]
	Since $v(T_n) \to v(T)$, $\|T_n\|\to \|T\|$, and $\|T\|>0$, if the desired inequality holds for
	\(T_n\), then we deduce that
	\[ \frac{v(T)}{\|T\|} = \lim_{n \to \infty} \frac{v(T_n)}{\|T_n\|} \geq v(J).\]
	Assume now that \(b>0\).  Since the quotient \(v(T)/\norm{T}\) is invariant under
	multiplication by a positive scalar, we divide by \(b\).  We obtain
	\[
	\frac1b T
	=
	\begin{pmatrix}
		a/b&1\\
		-c/b&-d/b
	\end{pmatrix}.
	\]
	Since \(0\le c/b\le1\), this matrix is of the normalized form \eqref{eq: T}. The result now follows for every $p>2$ by combining Lemma \ref{lem: lower bound} and Lemma \ref{lem: upper bound}.
\end{proof}

\section{Proof of Lemma~\ref{lem: lower bound}
	and Lemma~\ref{lem: upper bound}}\label{sec:lower-convexity}

We first prove the lower bound for the numerical radius.

\begin{proof}[Proof of Lemma~\ref{lem: lower bound}]
	By Lemma \ref{lem:numerical-radius-formula}, the numerical radius of an operator $T$ defined as in (\ref{eq: T}) is given by
	\[
	v(T)=\max\left\{
	\max_{0\le t\le1}
	\frac{\abs{a-dt^p}+t-(1-c)t^{p-1}}{1+t^p},
	\ 
	\max_{0\le t\le1}
	\frac{\abs{d-at^p}+\abs{-(1-c)t+t^{p-1}}}{1+t^p}
	\right\},
	\]
	where we have used $p>2$ to remove the absolute value. We evaluate both maxima at \(t=t_0\) and reformulate in terms of $L_0$, $M_0$, $N_0$, and $R_0$. For the first term, using
\[
a=\alpha-\beta,
\qquad
d=\alpha+\beta,
\]
we obtain
\[
a-dt_0^p
=
(\alpha-\beta)-(\alpha+\beta)t_0^p
=
\alpha(1-t_0^p)-\beta(1+t_0^p).
\]
Hence
\[
\frac{\abs{a-dt_0^p}}{1+t_0^p}
=
\abs{
	\alpha\frac{1-t_0^p}{1+t_0^p}-\beta
}
=
\abs{\alpha N_0-\beta}.
\]
Moreover,
\[
\frac{t_0-(1-c)t_0^{p-1}}{1+t_0^p}
=
\frac{t_0-t_0^{p-1}}{1+t_0^p}
+
c\frac{t_0^{p-1}}{1+t_0^p}
=
R_0+cL_0.
\]
Therefore,
\[
\frac{\abs{a-dt_0^p}+t_0-(1-c)t_0^{p-1}}{1+t_0^p}
=
R_0+cL_0+\abs{\alpha N_0-\beta}.
\]
For the second term,
\[
d-at_0^p
=
(\alpha+\beta)-(\alpha-\beta)t_0^p
=
\alpha(1-t_0^p)+\beta(1+t_0^p),
\]
and hence
\[
\frac{\abs{d-at_0^p}}{1+t_0^p}
=
\abs{
	\alpha\frac{1-t_0^p}{1+t_0^p}+\beta
}
=
\abs{\alpha N_0+\beta}.
\]
Moreover,
\begin{align*}
	\frac{\abs{(1-c)t_0-t_0^{p-1}}}{1+t_0^p}
	\geq
	\frac{(1-c)t_0-t_0^{p-1}}{1+t_0^p} 
	=
	\frac{t_0-t_0^{p-1}}{1+t_0^p}
	-
	c\frac{t_0}{1+t_0^p} 
	=
	R_0-cM_0.
\end{align*}
Therefore,
\[
\frac{\abs{d-at_0^p}+\abs{-(1-c)t_0+t_0^{p-1}}}{1+t_0^p}
\ge
R_0-cM_0+\abs{\alpha N_0+\beta}.
\]
Taking the maximum of the two lower bounds gives the result. \qedhere
\end{proof}

We next prove the upper bound for the normalized sign class, assuming the upper
triangular estimate of Lemma~\ref{lem: upper triangular}.

\begin{proof}[Proof of Lemma~\ref{lem: upper bound}]
	Assume first that \(0<c\le1\).  Write
	\[
	T=(1-c)J+cS,
	\qquad
	S=
	\begin{pmatrix}
		a/c & 1\\
		0 & -d/c
	\end{pmatrix}.
	\]
	By convexity of the operator norm,
	\[
	\norm{T}
	\le
	(1-c)\norm{J}+c\norm{S}
	=
	(1-c)+c\norm{S}.
	\]
	Applying Lemma~\ref{lem: upper triangular} to \(S\) and the homogeneity property (\ref{eq: homogeneity of H}), we deduce that
	\[
	\norm{S}
	\le
	\frac{v(J)+H(a/c,1,d/c)}{v(J)} = 1+ \frac{H(a,c,d)}{c v(J)}.
	\]
	Consequently,
	\begin{align*}
		\norm{T}
		\leq
		(1-c)+
		c\left(1 + \frac{H(a,c,d)}{c v(J)}
		\right)
		=
		1+\frac{H(a,c,d)}{v(J)}.
	\end{align*}
	If \(c=0\), we argue by approximation.  For \(0<c_n\le1\) with \(c_n\downarrow0\), set
	\[
	T_n :=
	\begin{pmatrix}
		a & 1\\
		-(1-c_n) & -d
	\end{pmatrix}.
	\]
	By the case already proved,
	\[
	\norm{T_n}
	\le
	\frac{v(J)+H(a,c_n,d)}{v(J)}.
	\]
	Since \(\|T_n\|\to \|T\|\) and \(H(a,c_n,d)\to H(a,0,d)\), letting
	\(n\to\infty\) gives
	\[
	\norm{T}
	\le
	\frac{v(J)+H(a,0,d)}{v(J)}.\qedhere
	\]
	
\end{proof}

\section{Proof of Lemma \ref{lem: upper triangular}}\label{sec:triangular}

We first record a simple norm criterion.

\begin{lemma}\label{lem:norm-criterion}
Let $A,B\ge0$ and $K\ge1$.  If
\[
 A\le K-1
 \qquad\text{and}\qquad
 B^p\le K^p-K^{p-1},
\]
then
\[
 \left\|\begin{pmatrix}A&1\\0&-B\end{pmatrix}\right\|\le K.
\]
\end{lemma}

\begin{proof}
For $A,B\ge0$, the norm of this operator may be computed on vectors with
nonnegative coordinates. Thus, a parametrization of the positive quadrant of the unit sphere of \(\ell_p^2\) by
\[
\frac{(1,s)}{(1+s^p)^{1/p}},
\qquad s\ge0,
\]
shows that
\begin{equation}\label{eq:norm-of-S}
 \left\|\begin{pmatrix}A&1\\0&-B\end{pmatrix}\right\|^p
 =\sup_{s\ge0}\frac{(A+s)^p+B^ps^p}{1+s^p}.
\end{equation}
Hence it is enough to show
\[
 (A+s)^p+B^ps^p\le K^p(1+s^p)
 \qquad\text{for all }s\ge0.
\]
Since \(A\le K-1\), we have
	\[
	(A+s)^p\le (K-1+s)^p.
	\]
	Now write
	\[
	\frac{K-1+s}{K}
	=
	\frac{K-1}{K}\cdot 1+\frac1K\cdot s.
	\]
	By convexity of \(t\mapsto t^p\),
	\[
	\left(\frac{K-1+s}{K}\right)^p
	\le
	\frac{K-1}{K}+\frac{s^p}{K}.
	\]
	Multiplying by \(K^p\), we get
	\[
	(K-1+s)^p
	\le
	K^{p-1}(K-1+s^p).
	\]
	Using the assumption on $B$ we obtain
	\begin{align*}
		(A+s)^p+B^ps^p&
		\leq
		K^{p-1}(K-1+s^p)+K^{p-1}(K-1)s^p 
		=
		K^{p-1}(K-1)+K^ps^p
		\\
		&\leq
		K^p(1+s^p).\qedhere
	\end{align*}
	
\end{proof}

We also need the following identities and inequalities, which we prove in Section \ref{sec:7}.

\begin{lemma}\label{lem:constants}
	The constants \(L_0,M_0,N_0,R_0\) and \(\eta\) satisfy:
	\begin{enumerate}[label=\textup{(\roman*)}]
		\item \(M_0-L_0=R_0\).
		\item \(M_0+L_0=\dfrac{R_0N_0}{\eta}\).
		\item $L_0=\frac{R_0}{2}\left(\frac{N_0}{\eta}-1\right)$, and $M_0=\frac{R_0}{2}\left(\frac{N_0}{\eta}+1\right).$
		\item \(\dfrac{L_0}{M_0}=t_0^{p-2}\).
		\item \(R_0\le N_0\).
		\item \(N_0^2\ge \eta\).
		\item \(\dfrac{R_0(1+N_0)}{\eta+N_0}\le t_0\).
	\end{enumerate}
\end{lemma}

\begin{proof}[Proof of Lemma~\ref{lem: upper triangular}]
	Fix \(n\ge0\) and suppose that \(A,B\ge0\) satisfy $H(A,1,B)\le n.$ Recall that $v(J)=R_0$. We will show that
	\begin{equation}\label{eq:levels}
	\left\|
	\begin{pmatrix}
		A&1\\
		0&-B
	\end{pmatrix}
	\right\|
	\le
	\frac{R_0+n}{R_0}.
	\end{equation}
	Let
	\[
	\alpha=\frac{A+B}{2},
	\qquad
	\beta=\frac{B-A}{2},
	\]
	and define
	\[
	\Phi:=\beta-\alpha N_0,
	\qquad
	\Psi:=\beta+\alpha N_0.
	\]
	The constraint \(H(A,1,B)\le n\) is equivalent to
	\begin{equation}\label{eq:constraint}
	\abs{\Phi}\le n-L_0,
	\qquad
	\abs{\Psi}\le n+M_0.
	\end{equation}
	In particular, if the region is nonempty, then \(n\ge L_0\). Solving for \(\alpha\) and \(\beta\), we get
	\[
	\alpha=\frac{\Psi-\Phi}{2N_0},
	\qquad
	\beta=\frac{\Phi+\Psi}{2}.
	\]
	Using \(A=\alpha-\beta\) and \(B=\alpha+\beta\), we can write $A$ and $B$ in terms of $\Phi$ and $\Psi$. Namely,
	\[
	A(\Phi,\Psi)=
	\frac12\left[
	\Psi\left(\frac1{N_0}-1\right)
	-
	\Phi\left(\frac1{N_0}+1\right)
	\right],
	\]
	and
	\[
	B(\Phi,\Psi)=
	\frac12\left[
	\Psi\left(\frac1{N_0}+1\right)
	-
	\Phi\left(\frac1{N_0}-1\right)
	\right].
	\]
	Since \(0<N_0<1\), both $(N_0^{-1} +1)$ and $(N_0^{-1} - 1)$ are positive. Hence, the formulas above show that \(A(\Phi,\Psi)\) and \(B(\Phi,\Psi)\) are increasing functions of \(\Psi\) and decreasing functions
	of \(\Phi\). Therefore for every feasible $(\Phi, \Psi)$ satisfying \eqref{eq:constraint}, the corresponding variables $A(\Phi,\Psi)$ and $B(\Phi,\Psi)$ are dominated by $A_*:=A(\Phi_*,\Psi_*)$ and $B_*:=B(\Phi_*,\Psi_*)$, where
	\[
	\Phi_*:=-(n-L_0),\qquad \Psi_*:=n+M_0.
	\] 
	Furthermore, by \eqref{eq:norm-of-S}, the norm of the triangular matrix is increasing
	in both \(A\) and \(B\). Therefore it is enough to prove the desired inequality (\ref{eq:levels}) for the corner $(A_*,B_*)$. A simple computation gives 
	\[
	A_* = \frac{n}{N_0}
	+
	\frac{M_0(1-N_0)-L_0(1+N_0)}{2N_0},
	\qquad
	B_* = \frac{n}{N_0}
	+
	\frac{M_0(1+N_0)-L_0(1-N_0)}{2N_0}.
	\]
	Using \(M_0-L_0=R_0\), we can write this corner as
	\[
	A_*=\widetilde A+\frac{n-L_0}{N_0},
	\qquad
	B_*=\widetilde B+\frac{n-L_0}{N_0},
	\]
	where
	\[
	\widetilde A:=\frac{(M_0+L_0)(1-N_0)}{2N_0},
	\qquad
	\widetilde B:=\frac{(M_0+L_0)(1+N_0)}{2N_0}.
	\]
	Hence, by the triangle inequality and \(R_0\le N_0\),
	\begin{align*}
		\left\|
		\begin{pmatrix}
			A_*&1\\
			0&-B_*
		\end{pmatrix}
		\right\|
		&\le
		\left\|
		\begin{pmatrix}
			\widetilde A&1\\
			0&-\widetilde B
		\end{pmatrix}
		\right\|
		+
		\frac{n-L_0}{N_0}
		\left\|
		\begin{pmatrix}
			1&0\\
			0&-1
		\end{pmatrix}
		\right\| \leq 
		\left\|
		\begin{pmatrix}
			\widetilde A&1\\
			0&-\widetilde B
		\end{pmatrix}
		\right\|
		+
		\frac{n-L_0}{R_0}.
	\end{align*}
	We claim that
	\begin{equation}\label{eq:special-base}
		\left\|
		\begin{pmatrix}
			\widetilde A&1\\
			0&-\widetilde B
		\end{pmatrix}
		\right\|
		\le
		1+\frac{L_0}{R_0}.
	\end{equation}
	Note that the result follows from this claim since for any $A$, $B\geq 0$ with $H(A,1,B)\leq n$ we have
	\[
	\left\|
	\begin{pmatrix}
		A&1\\
		0&-B
	\end{pmatrix}
	\right\|
	\leq 
	\left\|
	\begin{pmatrix}
		A_*&1\\
		0&-B_*
	\end{pmatrix}
	\right\|
	\le
	1+\frac{L_0}{R_0}+\frac{n-L_0}{R_0}
	=
	\frac{R_0+n}{R_0}.
	\]
	Thus, it remains to prove \eqref{eq:special-base}. Observe that by Lemma~\ref{lem:constants}, we have
	
	\[
	K:
	=
	1+\frac{L_0}{R_0}
	\overset{\textup{(i)}}{=}
	\frac{M_0}{R_0}
	\overset{\textup{(iii)}}{=}
	\frac12\left(\frac{N_0}{\eta}+1\right).
	\]
	By Lemma~\ref{lem:norm-criterion}, it suffices to show
	\[
	\widetilde A\le K-1
	\quad \mbox{and} \quad
	\widetilde{B}^p\le K^p-K^{p-1}.
	\]
	For the first inequality, again using Lemma~\ref{lem:constants},
	\[
	\widetilde A
	=
	\frac{(M_0+L_0)(1-N_0)}{2N_0}
	\;\overset{\textup{(ii)}}{=}\;
	\frac{R_0(1-N_0)}{2\eta}
	\;\overset{\textup{(v)}}{\le}\;
	\frac{N_0(1-N_0)}{2\eta}
	\;\overset{\textup{(vi)}}{\le}\;
	\frac{N_0-\eta}{2\eta}
	\;\overset{\textup{(iii)}}{=}\;
	\frac{L_0}{R_0} \; = \; K-1.
	\]
	For the second inequality, first observe that
	\[
	\frac{\widetilde B}{K}
	=
	\frac{(M_0+L_0)(1+N_0)}{2N_0K}
	\;\overset{\textup{(ii)}}{=}\;
	\frac{R_0(1+N_0)}{\eta+N_0}
	\;\overset{\textup{(vii)}}{\le}\;
	t_0.
	\]
	Therefore,
	\[
	\left(\frac{\widetilde B}{K}\right)^p
	\le
	t_0^p
	\le
	t_0^{p-2}
	\;\overset{\textup{(iv)}}{=}\;
	\frac{L_0}{M_0}
	\;\overset{\textup{(i)}}{=}\;
	1-\frac1K.
	\]
	Multiplying by \(K^p\), we obtain
	\[
	\widetilde{B}^p\le K^p-K^{p-1}.
	\]
	Thus Lemma~\ref{lem:norm-criterion} gives \eqref{eq:special-base}, and the proof is complete.
\end{proof}

\section{Proof of Lemma \ref{lem:constants}}\label{sec:7}

In order to prove Lemma~\ref{lem:constants}, we introduce the change of variables
\[
t^p=e^{-2z},
\qquad z\ge0.
\]
Using $\eta=(p-2)/p$, we compute
\[
M(t)=\frac{e^{\eta z}}{2\cosh z},
\qquad
L(t)=\frac{e^{-\eta z}}{2\cosh z},
\qquad
N(t)=\tanh z,
\qquad 
R(t)=M(t)-L(t)=\frac{\sinh(\eta z)}{\cosh z}.
\]
Let \(z_0\) be defined by \(t_0^p=e^{-2z_0}\), where $t_0$ is the maximizer of $R(t)$.

\begin{proof}[Proof of Lemma~\ref{lem:constants}]
	First, we record a useful identity for the maximizer $z_0$. Differentiating $R(t)$ gives that $z_0$ is characterized by
	\[
	\eta\cosh(\eta z_0)\cosh z_0
	=
	\sinh(\eta z_0)\sinh z_0.
	\]
	Equivalently,
	\begin{equation}\label{eq:critical}
		\tanh(\eta z_0)\tanh z_0=\eta.
	\end{equation}
Note that (i) is immediate from the definition of $L_0$, $M_0$, and $R_0$. Next, since
\[
 N_0=\tanh z_0,
 \qquad
 M_0+L_0=\frac{\cosh(\eta z_0)}{\cosh z_0},
 \qquad
 R_0=\frac{\sinh(\eta z_0)}{\cosh z_0},
\]
the identity (\ref{eq:critical}) gives
\[
 \eta(M_0+L_0)=N_0R_0.
\]
Dividing by $\eta$ proves (ii). Moreover, combining the above identity with $M_0-L_0=R_0$ gives (iii). Furthermore, (iv)  follows directly from
\[
 \frac{L(t)}{M(t)}=t^{p-2}.
\]
Next, observe that for any $0\leq t\leq 1$ we have
\[
 N(t)-R(t)=\frac{1-t^p-t+t^{p-1}}{1+t^p}
 =\frac{(1-t)(1+t^{p-1})}{1+t^p}\geq 0.
\]
Taking $t=t_0$ proves (v). The inequality (vi) follows from the identity (\ref{eq:critical}) and the fact that the function $\tanh$ is increasing. Indeed, since $\eta <1$ notice that 
\[ \eta = \tanh(\eta z_0)\tanh z_0 \leq \tanh^2 z_0 = N_0^2.\]
It remains to prove (vii).  Using the hyperbolic expressions above, we get
\[
\frac{R_0(1+N_0)}{\eta+N_0}
=
\frac{\frac{\sinh(\eta z_0)}{\cosh z_0}(1+\tanh z_0)}
{\eta+\tanh z_0}.
\]
The critical identity \eqref{eq:critical} gives
\[
\eta+\tanh z_0
=
\tanh z_0\bigl(1+\tanh(\eta z_0)\bigr).
\]
Hence
\[
\frac{R_0(1+N_0)}{\eta+N_0}
=
\frac{\sinh(\eta z_0)}{\cosh z_0}
\frac{1+\tanh z_0}
{\tanh z_0(1+\tanh(\eta z_0))}.
\]
Using
\[
1+\tanh z=\frac{e^z}{\cosh z},
\]
we obtain
\[
\frac{1+\tanh z_0}{1+\tanh(\eta z_0)}
=
\frac{e^{z_0}\cosh(\eta z_0)}
{e^{\eta z_0}\cosh z_0}
=
e^{(1-\eta)z_0}\frac{\cosh(\eta z_0)}{\cosh z_0}.
\]
Thus
\[
\frac{R_0(1+N_0)}{\eta+N_0}
=
e^{(1-\eta)z_0}
\frac{\sinh(\eta z_0)\cosh(\eta z_0)}
{\sinh z_0\cosh z_0}
=
e^{(1-\eta)z_0}
\frac{\sinh(2\eta z_0)}{\sinh(2z_0)}.
\]
Since \(0<\eta<1\),
\[
\frac{\sinh(2\eta z_0)}{\sinh(2z_0)}
=
e^{-2(1-\eta)z_0}
\frac{1-e^{-4\eta z_0}}{1-e^{-4z_0}}
\le
e^{-2(1-\eta)z_0}.
\]
Consequently,
\[
\frac{R_0(1+N_0)}{\eta+N_0}
\le
e^{-(1-\eta)z_0}.
\]
Finally, since \(1-\eta=2/p\) and \(t_0^p=e^{-2z_0}\), we have
\[
e^{-(1-\eta)z_0}=e^{-2z_0/p}=t_0.
\]
This proves (vii), and the proof is complete.
\end{proof}

\section*{Acknowledgements}
The author thanks Alicia Quero for helpful discussions and comments related to this work.

\bibliographystyle{plain}
\bibliography{bib}

\end{document}